\documentclass[12pt]{amsart}
\usepackage{graphicx}
\usepackage{amsmath}
\usepackage{amsfonts}
\usepackage[symbol]{footmisc}%

\newtheorem{thm}{Theorem}[section]

\newtheorem{lem}[thm]{Lemma}
\newtheorem{prop}[thm]{Proposition}
\theoremstyle{definition}
\newtheorem{defn}[thm]{Definition}
\theoremstyle{remark}
\newtheorem{rem}[thm]{Remark}
\numberwithin{equation}{section}
\newcommand{\Real}{\mathbb R}
\newcommand{\Nat}{\mathbb N}

\begin{document}

\title[]{A simplified proof of the o-minimal Whitney Extension Theorem}%
\author{Beata Kocel-Cynk, Wies{\l}aw Paw{\l}ucki and Anna Valette}%

\address{(B. Kocel-Cynk) Instytut Matematyki Politechniki Krakowskiej, ul. Warszawska 24, 31-155 Kraków, Poland}%
\email{beata.kocel-cynk{@}pk.edu.pl}%

\address{(W. Pawłucki) Instytut Matematyki Uniwersytetu Jagiellońskiego, ul. Prof. St. Łojasiewicza 6, 30-048 Kraków, Poland}%
\email{wieslaw.pawlucki{@}im.uj.edu.pl}%

\address{(A. Valette) Katedra Teorii Optymalizacji i Sterowania Uniwersytetu Jagiellońskiego, ul. Prof. St. Łojasiewicza 6, 30-048 Kraków, Poland}%
\email{anna.valette{@}uj.edu.pl}%

\subjclass{Primary 14P15. Secondary 03C64, 32B20, 46E35}%
\keywords{O-minimality, Whitney extension theorem, $\varLambda_p$-regular mapping, $\varLambda_p$-regular cell}%

\date{9 February, 2026
$\phantom{......................................................................................................................}$
Research financed by the Narodowe Centrum Nauki (Poland); grant no. 2024/53/B/ST1/01147 (W. Pawłucki) and grant no. 2021/43/B/ST1/02359 (A. Valette)}
% ----------------------------------------------------------------
\begin{abstract} We give a proof of the o-minimal version of the Whitney Extension Theorem simplified as compared to the previous ones. A new simplifying ingredient is a definable variant of Urysohn's lemma for the class $\mathcal{C}^q$ (see Section 3).
\end{abstract}
\maketitle
% ----------------------------------------------------------------
\section{Introduction}

In this article we give a simplified proof of the following o-minimal version of the classical Whitney Extension Theorem [W], earlier proved in [KP1], [KP2] and in [T] (see also [PR]):

\medskip
\begin{thm} \, Given an o-minimal extension of the field of real numbers $\Real$, \, let $E$ be a definable closed subset of $\Real^n$ and let $p$ and $q$ be positive integers such that $p\leq q$. Let
$$
F(x, X) = \sum_{|\alpha|\leq p}\frac{1}{\alpha!}F^{\alpha}(x)X^{\alpha} \qquad (X = (X_1,\dots, X_n))
$$
be a definable $\mathcal{C}^p$-Whitney field on $E$. $($Definability of $F$ means that all $F^{\alpha}$ are definable functions.$)$

\medskip
Then there exists a definable $\mathcal{C}^p$-function $f: \Real^n\longrightarrow \Real$, \, of class $\mathcal{C}^q$ on $\Real^n \setminus E$ such that $D^{\alpha}f =F^{\alpha}$ on $E$ whenever $\alpha\in \mathbb{N}^n$, $|\alpha|\leq p$.
\end{thm}

\bigskip
As in [KP1], [KP2] and [T], the main tool of the proof are stratifications into $\varLambda_p$-regular cells and the extension procedure when $E$ is a $\varLambda_p$-regular cell. A simplifying ingredient of the present proof is  a definable version of Urysohn's lemma, which appeared in our paper [K-CPV] (see Section 3).

\bigskip
Throughout the article, $p$ and $q$ will denote positive integers, $p\leq q$.

\bigskip
\section{$\varLambda_p$-regular maps and stratifications into $\varLambda_p$-regular cells}

We recall after [KP1] (see also [KP2], [T] and [PR]), the definition of $\varLambda_p$-regular maps.

\medskip
\begin{defn} A  map $f:\varOmega\longrightarrow \Real^m$ defined on an open subset $\varOmega$ of $\Real^n$ is called \emph{$\varLambda_p$-regular} if it is of class $\mathcal{C}^p$ and

$$
\forall\, \alpha\in \mathbb{N}^n \, x\in \varOmega: \, 1\leq |\alpha|\leq p \, \Longrightarrow \, |D^{\alpha}f(x)|\leq \frac{C}{d(x,\partial \varOmega)^{|\alpha|-1}},
$$

\noindent
where $C$ is a positive constant.

In this definition, we adopt the convention $d(x,\emptyset):= 1$, to cover the case $\partial\varOmega = \emptyset$.
\end{defn}

\medskip
\begin{rem} In particular, all the first order partial derivatives $\frac{\partial f}{\partial x_j}$ are bounded. It follows that if $A$ is a  subset of $\varOmega$ which is \emph{quasi-convex} (i.e. there exists a positive constant $M$ such that for each pair of points $a, b\in A$ there exists a rectifiable arc $\gamma: [0, 1]\longrightarrow A$ such that $\gamma(0)=a$, $\gamma(1)=b$ and $|\gamma|\leq M|a - b|$), then $f|A$ is Lipschitz, hence it admits a continuous extension $\overline{f|A}: \overline{A}\longrightarrow \Real^m$ to the closure $\overline{A}$.
\end{rem}

\medskip
\begin{defn}
$\phantom{.......................................................................................................................}$

\medskip
 We say that $S$ is an \emph{open $\varLambda_p$-regular cell} in $\Real^n$ if
\begin{equation}
\text{$S$ is any open interval in $\Real$, when $n = 1$;}
\end{equation}
\begin{equation}
S = \{(x', x_n): x'\in T, \psi_1(x') < x_n < \psi_2(x')\},
\end{equation}
where $x' = (x_1,\dots, x_{n-1})$, $T$ is an open $\varLambda_p$-regular cell in $\Real^{n-1}$ and every $\psi_i$ \, $(i\in \{1, 2\})$ \, is either a definable $\varLambda_p$-regular function on $T$ with values in $\Real$, or identically equal to $-\infty$, or identically equal to $+\infty$, and $\psi_1(x') < \psi_2(x')$, for each $x'\in T$, when $n > 1$.

\medskip
Extending the above definition, we say that $S$ is an \emph{$m$-dimensional $\varLambda_p$-regular cell} in $\Real^n$, where $m\in \{0,\dots, n-1\}$, if
\begin{equation}
S = \{(u, w): \, u\in T, w = \varphi(u)\},
\end{equation}
where $u = (x_1,\dots, x_m)$, $w = (x_{m+1},\dots, x_n)$, \, $T$ is an open $\varLambda_p$-regular cell in $\Real^m$, and $\varphi: T\longrightarrow \Real^{n-m}$ is a definable $\varLambda_p$-regular mapping.
\end{defn}

\medskip
\begin{rem}
It is elementary that if $S$ is an $m$-dimensional $\varLambda_p$-regular cell in $\Real^n$  as in (2.3) and $M\geq 0$ is a Lipschitz constant of the mapping $\varphi$, then putting $L := 1/\sqrt{1 + M^2}$, we have

\begin{equation}
\forall x = (u, w)\in T\times \Real^{n-m}: \, L|w - \varphi(u)| \leq d(x, S) \leq |w - \varphi(u)|
\end{equation}

and

\begin{equation}
\forall \, x\in \Real^n\setminus (T\times \Real^{n-m}): \, d(x, S) \geq L d(x, \partial S).
\end{equation}
\end{rem}

\medskip
\begin{rem} It follows from Remark 2.2 by induction that every $\varLambda_p$-regular cell is quasi-convex and it is Lipschitz in the sense that each of the functions $\psi_i$ in (2.2), if finite, as well as the map $\varphi$ in (2.3), are Lipschitz. Besides, every $\varLambda_p$-regular cell in $\Real^n$ is a definable connected $\mathcal C^p$-submanifold of $\Real^n$.
\end{rem}

\medskip
\begin{defn}

\medskip
Let us recall that a (\emph{definable}) $\mathcal C^p$-\emph{stratification} of a (definable) subset $E$ of $\Real^n$ is a finite decomposition $\mathcal S$ of $E$ into (definable) connected $\mathcal C^p$-submanifolds of $\Real^n$, called \emph{strata}, such that for each stratum $S\in \mathcal S$, its \emph{boundary} in $E$; i.e. $\partial_ES := (\overline{S}\setminus S)\cap E$ is the union of some strata of dimensions $< \dim S$. If $A_1,\dots, A_k$ \, $(k\in \Nat)$ \, are subsets of $E$, we call a stratification $\mathcal S$ \emph{compatible with the subsets $A_1,\dots, A_k$}, if each $A_j$ is a union of some strata.
\end{defn}

\begin{thm}[\lbrack P1, Proposition 4 or P2, Theorem 3\rbrack]
Given any finite number $A_1,\dots, A_k$ of definable subsets of a definable subset $E$ of $\Real^n$, there exists a $\mathcal C^p$-stratification $\mathcal S$ of $E$ compatible with the sets $A_1,\dots, A_k$ and such that each stratum $S\in \mathcal S$, after an orthogonal linear change of coordinates in $\Real^n$, is a $\varLambda_p$-regular cell in $\Real^n$.
\end{thm}

\begin{rem} By [P2], as orthogonal linear changes of coordinates in the above theorem it suffices to take permutations of Cartesian coordinates.
\end{rem}

\bigskip
\section{Definable Urysohn's lemma for the class $\mathcal{C}^q$}

\medskip
Definable Urysohn's lemma for the class $\mathcal{C}^q$ (Theorem 3.2) is a basic new ingredient in our proof of Theorem 1.1.

\medskip
\begin{defn}
$\phantom{.............................................................................................................................................}$

\medskip
Let $Z$ be a closed definable subset of $\Real^n$ and let $W$ be a definable subset of $\Real^n\setminus Z$. We will consider the following open neighborhoods of $W$ in $\Real^n$

$$
G_\eta(W, Z) :=\{x\in \Real^n\setminus Z: \, d(x, W) < \eta d(x, Z)\},
$$
where $\eta > 0$.
\end{defn}

\medskip
\begin{thm}[\lbrack K-CPV, Proposition 3.9 \rbrack] Let $Z$ be a definable closed subset of $\Real^n$ and let $W$ be a definable subset of $\Real^n\setminus Z$. Then for any positive integer $q$ and for any $\eta > 0$ there exists a definable $\mathcal{C}^q$-function  $\omega: \Real^n\setminus Z\longrightarrow [0, 1]$ such that

\begin{equation}
\forall \, x\in \Real^n\setminus Z, \, \alpha\in \mathbb{N}^n: \quad |\alpha|\leq q \Longrightarrow |D^\alpha \omega(x)|\leq \frac{C}{d(x, Z)^{|\alpha|}},
\end{equation}
\begin{equation}
\text{$\omega \equiv 1$ on $G_\rho(W, Z)$, for some $\rho\in (0, \eta)$, and}
\end{equation}
\begin{equation}
\overline{\{x\in \Real^n\setminus Z: \omega(x) \ne 0\}}\cap \big(\mathbb{R}^n\setminus Z\big)\subset G_\eta(W, Z).
\end{equation}
\end{thm}

\bigskip
\section{$\mathcal{C}^p$-Whitney fields}

\medskip
Let $A$ be a locally closed subset of $\Real^n$. By $\mathcal{C}(A)$ we denote the algebra of continuous functions on $A$ with values in $\Real$. A  \emph{$\mathcal{C}^p$-Whitney field} on $A$ is a polynomial

$$
F(u, X) = \sum_{|\alpha|\leq p}\frac{1}{\alpha!} F^{\alpha}(u)X^{\alpha}\in \mathcal{C}(A)[X] = \mathcal{C}(A)[X_1,\dots,X_n],
$$

which fulfills the following condition

\begin{enumerate}
\item[($\ast$)] {\, for each $c\in A$ and for each $\beta\in \mathbb{N}^n$ such that $|\beta|\leq p$

$$
F^{\beta}(a)- \sum_{|\alpha|\leq p-|\beta|} \frac{1}{\alpha!} F^{\alpha + \beta}(b)(a - b)^{\alpha} = o(|a - b|^{p-|\beta|}), \quad\text{when $A\ni a\rightarrow c$, $A\ni b\rightarrow c$}.
$$}
\end{enumerate}

We call such a $\mathcal{C}^p$-Whitney field \emph{definable}, if $A$ is definable and all the functions $F^{\alpha}$ are definable.

\bigskip
\begin{rem} If $F$ is a $\mathcal{C}^p$-Whitney field on a locally closed subset $A\subset \Real^n$, then its \emph{restriction} $F|B$ to a locally closed subset $B\subset A$, defined by
$F|B(u, X) = F(u, X)$, for any $u\in B$, is a $\mathcal{C}^p$-Whitney field on $B$.
\end{rem}

\medskip
Let $\pi_p: \mathcal{C}(A)[X]\longrightarrow \mathcal{C}(A)[X]$ denote the natural projection

$$
\pi_p(\sum_{\alpha} \frac{1}{\alpha!}F^{\alpha}X^{\alpha}):= \sum_{|\alpha|\leq p} \frac{1}{\alpha!}F^{\alpha}X^{\alpha}
$$
onto the space of polynomials of degree $\leq p$.

\medskip
\begin{rem} The set $\mathcal{E}^p(A)$ of all $\mathcal{C}^p$-Whitney fields on a locally closed subset $A\subset \mathbb{R}^n$ with its natural addition and the multiplication defined by $FG := \pi_p(FG)$ is an $\Real$-algebra.
\end{rem}

\bigskip
\begin{rem}It is also natural to define the composition of $\mathcal{C}^p$-Whitney fields as follows.\newline Let $F_1,\dots,F_m\in \mathcal{E}^p(A)$, where $A\subset \mathbb{R}^n$ is locally closed and let $H\in \mathcal{E}^p(B)$, where $B\subset \mathbb{R}^m$ is locally closed. Assume that $(F_1^0,\dots,F_m^0)(A)\subset B$. Put $F:=(F_1,\dots,F_m)$. Then the composition $H\circ F$ defined by the formula

$$
(H\circ F)(u, X) := \pi_p[ H(F_1^0(u),\dots, F_m^0(u), F_1(u, X)-F_1^0(u),\dots,F_m(u, X)-F_m^0(u))],
$$
for any $u\in A$, is a $\mathcal{C}^p$-Whitney field.
\end{rem}

\medskip
\begin{rem} It follows from the Taylor formula that if $f: U\longrightarrow \Real$ is a (definable) $\mathcal{C}^p$-function on a (definable) open subset $U \subset \Real^n$, then the polynomial

$$
Tf(u, X)= T_u^pf(X):= \sum_{|\alpha|\leq p}\frac{1}{\alpha!}D^{\alpha}f(u)X^{\alpha},\quad\text{where $u\in U$},
$$
is a (definable) $\mathcal{C}^p$-Whitney field on $U$.
\end{rem}

\bigskip
\begin{rem} ([G, pages 87-88]) Let $k$ and $n$ be integers such that $1\leq k\leq n$. Let $D$ be an open subset of $\Real^k$ treated as a subset of $\Real^n$ by the injection $\Real^k\ni v\longmapsto (v, 0)\in \Real^k\times \Real^{n-k}=\Real^n$. Then every $\mathcal{C}^p$-Whitney field

$$
F(v, X) = F(v, V, W) = \sum_{|\alpha| + |\beta|\leq p}\frac{1}{\alpha!\beta!}F^{(\alpha,\beta)}(v)V^\alpha W^\beta
$$
on $D$, where $\alpha\in \mathbb{N}^k$, $\beta\in \mathbb{N}^{n-k}$, $V = (X_1,\dots,X_k)$ and $W = (X_{k+1},\dots,X_n)$, can be identified with a polynomial

$$
\tilde{F}(v, W)= \sum_{|\beta|\leq p}\frac{1}{\beta!}F^{(0,\beta)}(v)W^\beta\in \mathcal{C}(D)[W],
$$
where, for each $\beta\in \mathbb{N}^{n-k}$ such that $|\beta|\leq p$, $F^{(0,\beta)}$ is a $\mathcal{C}^{p-|\beta|}$-function on $D$ such that

\begin{enumerate}
\item[($\ast\ast$)]{\qquad ${D^{\alpha}}F^{(0,\beta)} = F^{(\alpha,\beta)}$, \quad for each $\alpha\in \mathbb{N}^k$ such that $|\alpha|\leq p - |\beta|$.}
\end{enumerate}
\end{rem}

\medskip
\begin{defn}
Two closed subsets $A$ and $B$ of $\Real^n$ are called \emph{simply separated} if there exists a positive constant $M$ such that, for each $x\in A$, $d(x, A\cap B)\leq Md(x, B)$ (compare [P1]).
\end{defn}

\bigskip
The following proposition plays an especially important role in our extension procedure, similarly as in the previous papers.

\bigskip
\begin{prop}[\lbrack KP1, Proposition 3 or KP2, Proposition 3\rbrack] Let $\varPhi:\varOmega\longrightarrow {\mathbb R}^n$ be a $\varLambda_p$-regular map defined on an open subset $\varOmega\subset {\mathbb R}^m$ and let $A$ be a closed quasi-convex subset of $\varOmega$ such that $\overline{A}$ and $\partial\varOmega$ are simply separated\footnote{Then by Remark 2.2, $\varPhi|A$ is Lipschitz, hence has a continuous extension $\overline{\varPhi|A}$ to the closure $\overline{A}$.}. Let $B$ be a locally closed subset of $\Real^n$ such that $\varPhi(A)\subset B$ and let $r: B\longrightarrow [0, +\infty)$ be a  function such that

$$
r(y)\leq C'd(y, \overline{\varPhi|A}(\overline{A}\setminus A)) \qquad\text{for any $y\in B$},
$$
where $C'$ is a positive constant. Let

$$
F(y, Y) = \sum_{|\kappa|\leq p}\frac{1}{\kappa!} F^{\kappa}(y)Y^{\kappa}
$$
be a $\mathcal{C}^p$-Whitney field on $B$ such that, for each $b\in \overline{\varPhi|A}(\overline{A}\setminus A)$, $F^\kappa(y) = o(r(y)^{p-|\kappa|})$,
when $B\ni y\rightarrow b$ and $|\kappa|\leq p$. Put

$$
G(x, X):= F\circ T\varPhi(x, X)= \sum_{|\sigma|\leq p}\frac{1}{\sigma!}G^\sigma(x)X^\sigma, \quad\text{where \, $x\in A$ and $X = (X_1,\dots,X_m)$}.
$$

\medskip
Then, for each $a\in \overline{A}\setminus A$, \, $G^\sigma(x) = o(r(\varPhi(x))^{p-|\sigma|})$, when $A\ni x\rightarrow a$ and $|\sigma|\leq p$.
\end{prop}

\medskip
\begin{lem} [\lbrack KP1, Lemma 6 or KP2, Lemma 3\rbrack] Let $\varOmega$ be an open subset of $\Real^n$, $a\in \overline{\varOmega}$ and $r:\varOmega\longrightarrow [0, +\infty)$.
Let $g, h: \varOmega\longrightarrow \Real$ be two $\mathcal{C}^p$-functions such that $D^{\kappa}g(x) = o(r(x)^{p-|\kappa|})$ and $D^{\kappa}h(x) = O(r(x)^{-|\kappa|})$, when $x\rightarrow a$, for any $\kappa\in \mathbb{N}^n$ such that $|\kappa|\leq p$.

\medskip
Then $D^{\kappa}(gh)(x) = o(r(x)^{p-|\kappa|})$, when $x\rightarrow a$, for any $\kappa\in \mathbb{N}^n$ such that $|\kappa|\leq p$.
\end{lem}

\begin{proof} Immediately from the Leibniz formula.
\end{proof}

\medskip
\section{The generic case: Whitney field flat except on one $\varLambda_q$-regular cell}

\bigskip
In this section we consider a special case, when our $\mathcal{C}^p$-Whitney field $F$ is zero except on one $\varLambda_q$-regular cell which is open in the set $E$.

\begin{prop} Let $\varLambda =\{x = (u,w)\in D\times \Real^{n-k}: \, w=\varphi(u)\}$ be a $k$-dimensional $\varLambda_q$-regular cell open in $E$, where $D$ is an open ${\varLambda}_q$-regular cell in $\Real^k$ and $\varphi=(\varphi_{k+1},\dots,\varphi_n): D\longrightarrow \Real^{n-k}$ is a $\varLambda_q$-regular map.

Let $F$ be a definable $\mathcal{C}^p$-Whitney field on $E$ which is flat $($i.e. zero$)$ on $Z:= E\setminus \varLambda$ and such that for each $\beta\in \mathbb{N}^{n-k}$, with $|\beta|\leq p$,  the restriction $F^{(0,\beta)}|\varLambda$ is of class $\mathcal{C}^q$.

\medskip
Then there exists a definable open neighborhood $V$ of $\varLambda$ such that \newline $V\subset (D\times \Real^{n-k})\setminus Z$ and a definable $\mathcal{C}^q$-function $\omega: V\longrightarrow [0, 1]$ such that
$\omega\equiv 1$ in a neighborhood of $\varLambda$,

$$
\overline{\{x\in V: \, \omega(x)\ne 0\}}\cap\bigl(D\times \Real^{n-k}\bigr)\subset V;
$$
$$
\forall \, x\in V, \, \alpha\in \mathbb{N}^n: \quad |\alpha|\leq q \Longrightarrow |D^\alpha \omega(x)|\leq \frac{C}{d(x, Z)^{|\alpha|}}
$$
where $C$ is a positive constant, and the function

$$
f(x)=f(u, w):= \sum_{|\beta|\leq p}\frac{1}{\beta!}F^{(0,\beta)}(u,\varphi(u))(w - \varphi(u))^{\beta} = F(\pi(x), x- \pi(x)),
$$

\noindent
defined for $x=(u, w)\in D\times \Real^{n-k}$, where $\pi(x)= \pi(u, w):= (u, \varphi(u))$,  is of class $\mathcal{C}^q$ in $D\times \Real^{n-k}$ and the function $f(x)\omega(x)$ extends to a $\mathcal{C}^p$-function on $\Real^n$ which is a $\mathcal{C}^p$-extension of the Whitney field $F$ and which is of class $\mathcal{C}^q$ on $\Real^n\setminus \partial\varLambda$.
\end{prop}

\medskip
\noindent
\begin{proof} Consider the following $\varLambda_q$-regular automorphism of the open set $D\times \mathbb{R}^{n-k}$:

$$
\varPhi: D\times \mathbb{R}^{n-k}\ni (u, w)\longmapsto (u, w + \varphi(u))\in D\times \mathbb{R}^{n-k}.
$$

\medskip
\noindent
Now let us shift our Whitney field $F|\varLambda$ by $\varPhi$ getting the following $\mathcal{C}^p$-Whitney field on $D\times\{0\}$:

$$
G((u, 0), (U, W)):=(F|\varLambda)\circ T\varPhi((u, 0), (U, W)) =
$$
$$
\sum_{|\alpha|+|\beta|\leq p}\frac{1}{\alpha!\beta!}F^{(\alpha,\beta)}(u,\varphi(u))U^{\alpha}\bigl(W + \sum_{1\leq |\gamma|\leq p}\frac{1}{\gamma!}D^{\gamma}\varphi(u)U^{\gamma}\bigr)^{\beta} \, {\rm mod}(U,W)^{p+1},
$$

\medskip
\noindent
where $\alpha, \gamma\in \mathbb{N}^k$ and $\beta\in \mathbb{N}^{n-k}$.

\medskip
By Remark 4.3, Remark 4.4, Remark 4.5 and our assumption that the functions $F^{(0,\beta)}$ are of class $\mathcal{C}^q$, this field extends to the following $\mathcal{C}^q$-function on $D\times \Real^{n-k}$

$$
\tilde{g}(x) = \tilde{g}(u, w) =\sum_{|\beta|\leq p}\frac{1}{\beta!}F^{(0,\beta)}(u,\varphi(u))w^{\beta}= \sum_{|\beta|\leq p}\frac{1}{\beta!}G^{\beta}(u)w^{\beta},
$$

\medskip
\noindent
which is a $\mathcal{C}^p$-extension of $G$ to $D\times \mathbb{R}^{n-k}$.

\medskip
\noindent
Since $F$ is flat on $Z\supset\overline{\varLambda}\setminus \varLambda$, by the Whitney regularity condition,
$$
\forall \, a\in \partial D: \quad F^{(\alpha, \beta)}(x) = o((d(x, Z))^{p-|\alpha|-|\beta|}),
$$
when $x\longrightarrow (a,\overline{\varphi}(a))\in \overline{\varLambda}\setminus \varLambda$ and $(\alpha, \beta)\in \mathbb{N}^k\times \mathbb{N}^{n-k}$ such that $|\alpha|+|\beta|\leq p$.

\medskip
\noindent
It follows by Proposition 4.7 that if $(\alpha, \beta)\in \mathbb{N}^k\times \mathbb{N}^{n-k}$ such that $|\alpha|+|\beta|\leq p$, then

$$
\forall \, a\in \partial D: \quad D^{\alpha}G^{\beta}(u) = o((d((u, \varphi(u)), Z))^{p-|\alpha|-|\beta|}) = o((d((u, 0), {\overline{\varPhi}}^{-1} (Z)))^{p-|\alpha|-|\beta|}),
$$
$\phantom{...............................................................................................}$when $u\longrightarrow a$.

\medskip
\noindent
It follows that given any positive constant $C$, then

$$
\forall \, a\in \partial D: \, D^{(\alpha,\beta)}\tilde{g}(u, w ) = o((d((u, w), {\overline{\varPhi}}^{-1} (Z)))^{p-|\alpha|-|\beta|}),\phantom{....................}
$$
$\phantom{...............................}$when $u\longrightarrow a$ and $|w| = d((u, w), D\times {0})\leq Cd((u, w), {\overline{\varPhi}}^{-1} (Z)$).

\bigskip
\noindent
Applying again Proposition 4.7, this time to the $\varLambda_q$-map \, $\varPhi^{-1}$ and to (the $\mathcal{C}^p$-Whitney field defined by) the function $\tilde{g}$ restricted to the set of the form

$$
\{(u, w)\in D\times \mathbb{R}^{n-k}: \, |w| < Cd((u, w), {\overline{\varPhi}}^{-1}(Z))\},
$$

we conclude that the function

$$
f(x) = f(u, w):= \sum_{|\beta|\leq p}\frac{1}{\beta!}G^{\beta}(u)(w - \varphi(u))^{\beta} = \sum_{|\beta|\leq p}\frac{1}{\beta!}F^{(0, \beta)}(u, \varphi(u))(w - \varphi(u))^{\beta}
$$

\medskip
\noindent
is a $\mathcal{C}^q$-extension of the $\mathcal{C}^p$-Whitney field $F|\varLambda$ to the set $D\times \mathbb{R}^{n-k}$ such that for any positive constant $C$, if $\kappa\in \mathbb{N}^n$ such that $|\kappa|\leq p$, then

$$
\forall (a,\overline{\varphi}(a))\in \partial \varLambda: \, D^{\kappa}f(x) = o((d(x, Z))^{p-|\kappa|}),\phantom{...........................}
$$
$\phantom{...........................................}$ when $x\rightarrow (a,\overline{\varphi}(a))$ and $d(x, \varLambda)\leq Cd(x, Z)$.

\bigskip
\noindent
Now, by Theorem 3.2 and Remark 2.4, there exists a definable $\mathcal{C}^q$-function $\omega: \mathbb{R}^n\setminus Z\longrightarrow [0, 1]$ such that for some constants $\rho, \eta\in (0, C)$ such that $\rho < \eta$

\medskip
$$
\forall \, x\in \mathbb{R}^n\setminus Z, \, \alpha\in \mathbb{N}^n: \quad |\alpha|\leq p \Longrightarrow |D^\alpha \omega(x)|\leq \frac{C}{d(x, Z)^{|\alpha|}},
$$

\medskip
$$
\forall \, x\in \mathbb{R}^n: \,\, d(x, \varLambda) \leq \eta d(x, Z) \Longrightarrow x\in D\times \mathbb{R}^{n-k},
$$

\medskip
$$
\overline{\{x\in \Real^n\setminus Z: \, \omega(x)\ne 0\}}\cap \big(\Real^n\setminus Z\big)\subset \{x\in {\mathbb{R}}^n: \,\, d(x, \varLambda) < \eta d(x, Z) \}
$$

and

$$
\forall \, x\in \mathbb{R}^n: \,\,  d(x, \varLambda) < \rho d(x, Z) \Longrightarrow \omega(x) = 1.
$$

\bigskip
Lemma 4.8 concludes the proof of Proposition 5.1.
\end{proof}

\bigskip
\section{Proof of Theorem 1.1}

\medskip
We argue by induction on $\dim E$. If $\dim E = 0$, then $E$ is a finite set and the proof reduces to gluing together a finite number of polynomials by a definable $\mathcal{C}^q$-partition of unity. Hence, assume that $\dim E = k >0$. The following lemma is an easy observation.

\begin{lem} If $E'\subset E$ is a closed subset of $E$, $g:\Real^n\longrightarrow \Real$ is a $\mathcal{C}^p$-extension of $F|E'$ which is $\mathcal{C}^q$ on $\Real^n\setminus E'$ and $h:\Real^n\longrightarrow \Real$ is a $\mathcal{C}^p$-extension of $F - Tg|E$ which is $\mathcal{C}^q$ on $\Real^n\setminus E$, then $h + g$ is  a $\mathcal{C}^p$-extension of $F$ which is $\mathcal{C}^q$ on $\Real^n\setminus E$.
\end{lem}

Combining Theorem 2.7 with the induction hypothesis, the above lemma and the induction on the number of strata of dimension $k$, we see that without any loss of generality we can assume that $E = Z\cup S$, where $Z\cap S = \emptyset$, $S$ is a $\varLambda_q$-regular cell of dimension $k$ (perhaps after some permutation of coordinates $x_1,\dots, x_n$) open in $E$ (so $Z$ is closed) and $F$ is flat on $Z$. Again by Theorem 2.7, there exists a definable closed subset $A$ of $S$ such that $\dim A < k$, \, $S\setminus A = \varLambda_1\cup\dots\cup\varLambda_r$, where $r\in\mathbb{N}^{\ast}$ and all $\varLambda_j$ are $\varLambda_q$-regular cells of dimension $k$ in the same system of coordinates and all the restrictions $F^{(0,\beta)}|\varLambda_j$ \, $(\beta\in \mathbb{N}^{n-k}, \, |\beta|\leq p)$ \, are of class $\mathcal{C}^q$.
By Lemma 6.1, combined with the induction hypothesis applied to $F|A$, where $q$ is replaced by $p+q$, in order to keep the property that all the restrictions $F^{(0,\beta)}|\varLambda_j$ \, $(\beta\in \mathbb{N}^{n-k}, \, |\beta|\leq p)$ \, are of class $\mathcal{C}^q$, we can assume that $F|A$ is flat. Let $F_j$ denote the $\mathcal{C}^p$-Whitney field on $E$ which is equal to $F$ on $\varLambda_j$ and zero on $E\setminus \varLambda_j$, for each $j\in\{1,\dots, r\}$. Applying Proposition 5.1 to each of $F_j$ we get a definable $\mathcal{C}^p$-extension $f_j\omega_j:\Real^n\longrightarrow \Real$ which is of class $\mathcal{C}^q$ outside $\overline{\varLambda_j}$. Clearly, $\sum_{j=1}^rf_j\omega_j$ is a required extension of $F$.

\section{Final remarks}
By some obvious modifications, the above proof can be generalized to the case, when the definable $\mathcal{C}^p$-Whitney field $F$ is defined not on a closed subset of $\Real^n$ but on a definable set $E$ which is locally closed; i.e. closed in some definable open subset $\varOmega$ of $\Real^n$. Then there exists a definable $\mathcal{C}^p$-extension $f:\varOmega\longrightarrow \Real$ of the field $F$ which is of class $\mathcal{C}^q$ on $\varOmega\setminus E$.

\medskip
The authors believe that the results of the article can be easily modified to be true when the field of real numbers $\Real$ is replaced by any real closed field, because the technique of $\varLambda_p$-regular stratifications, in particular Theorem 2.7, holds true in this more general case.
\bigskip
\section{Statements and Declarations}

\medskip
On behalf of the authors, the corresponding author states that there is no conflict of interest.

\medskip
Data sharing not applicable to this article as no datasets were generated or analysed during the current study.

\bigskip
{\bf Acknowledgement.} The authors thank the anonymous referees whose comments considerably improved the first version of the paper.

\bigskip
\centerline{\bf References}

\bigskip

\medskip
[K-CPV] \, B. Kocel-Cynk, W. Pawłucki and A. Valette; \emph{Semialgebraic Calder{\'o}n-Zygmund Theorem on regularization of the distance function}; Math. Ann. {\bf 390 no. 2} (2024), 1863-1883.

\medskip
[G] \, G. Glaeser; \emph{{\'E}tude de quelques alg{\'e}bres tayloriennes}, J. Anal. Math. {\bf 6} (1958), 1-124.

\medskip
[KP1] \, K. Kurdyka and W. Pawłucki; \emph{Subanalytic version of Whitney's extension theorem}; Studia Math. {\bf 124} (1997), 269-280.

\medskip
[KP2] \, K. Kurdyka and W. Pawłucki; \emph{O-minimal version of Whitney's extension theorem}; Studia Math. {\bf 224} (2014), 81-96.

\medskip
[P1]  \, W. Pawłucki; \emph{A decomposition of a set definable in an o-minimal structure into perfectly situated sets}; Ann. Polon. Math. {\bf 79, no 2} (2002), 171-184.

\medskip
[P2] \, W. Pawłucki; \emph{Lipschitz cell decomposition in o-minimal structures.I}; Illinois J. Math. {\bf 52} (2008), 1045-1063.

\medskip
[PR] \, A.Parusiński, A. Rainer; \emph{Uniform extension of definable $C^{m,\omega}$-Whitney jets}; Pacific J. Math. {\bf 330 no. 2} (2024) 317-353.

\medskip
[T] \, A. Thamrongthanyalak; \emph{Whitney's extension theorem in o-minimal structures}; Ann. Polon. Math. {\bf 119 no. 1} (2017), 49-67.

\medskip
[W] \, H. Whitney; \emph{Analytic extension of differentiable functions defined in closed sets}; Trans. Amer. Math. Soc. {\bf 36, no. 1} (1934), 63-89.

\end{document}